\documentclass[11pt,a4paper]{article}
\usepackage[latin1]{inputenc}
\usepackage[english]{babel}
\usepackage{amsmath}
\usepackage{amsfonts}
\usepackage{amssymb}
\usepackage{graphicx}
\usepackage{natbib}
\usepackage{etoolbox}
\usepackage{hyperref}

\setcitestyle{authoryear}
\setcitestyle{round}
\makeatletter
\newcommand \bibstyle@comma{\bibpunct(),a,,}
\newcommand \bibstyle@semicolon{\bibpunct();a,,}
\makeatother
\pretocmd \citet{\citestyle{comma}}\relax \relax
\pretocmd \Citet{\citestyle{comma}}\relax \relax
\pretocmd \citep{\citestyle{semicolon}}\relax \relax
\pretocmd \Citep{\citestyle{semicolon}}\relax \relax
\newtheorem{theorem}{Theorem}
\newtheorem{definition}{Definition}
\newtheorem{proposition}{Proposition}

\newtheorem{remark}{Remark}
\newtheorem{corollary}{Corollary}
\newtheorem{example}{Example}
\newenvironment{proof}[1][Proof]{\textbf{#1.} }{\ \hfill \rule{0.5em}{0.5em}}
\author{Pierpaolo Uberti\thanks{DIEC Department of Economics, University of Genova, Italy, \texttt{uberti@economia.unige.it}}}
\title{A note on normality of $\sqrt{2}$ in base 2}

\begin{document}
\maketitle

\begin{abstract}
In this paper we study the property of normality of a number in base 2. A simple rule that associates a vector to a number is presented and the property of normality is stated for the vector associated to the number. The problem of testing a number for normality is shown to be equivalent to the test of geometrical properties of the associated vector. The paper provides a general approach for normality testing and then applies the proposed methodology to the study of particular numbers. The main result of the paper is to prove that an infinite class of numbers is normal in base 2. As a further result we prove that the irrational number $\sqrt{2}$ is normal in base 2.   
\medskip

%\textbf{Keywords}: 
\end{abstract}
\newpage
\section{Introduction}

Given an integer $b \geq 2$, a b-normal number (or a normal number) is a number whose b-ary expansion is such that any preassigned sequence of length $k \geq 1$ occurs at the expected frequency $\frac{1}{b^k}$. A number that is normal for every choice of a base $b$ is said to be absolutely normal. It is known since 1909 (see \cite{nla1}) that almost all real numbers are normal in every base b.

The concept of normal number in a given base $b$ is well known in mathematics. In base $b = 2$ the sequence of the digits of a normal number is equivalent to the sequence of heads and tails when flipping a coin; the frequency of heads and tails (or zeros and ones) is expected to converge to $\frac{1}{2}$ if we continue to flip the coin (or we consider the complete sequence of digits of an irrational normal number). %A number is said to be absolutely normal if normal in every base b.
 
A normal number in a given base b can be interpreted as the sequence of the outcomes when sampling from a uniform random variable. The interest in numbers with random behavior relates to the applications of random number generators in gambling, lotteries, computer simulation, cryptography and many other areas. Numbers, if proved to be normal, have their practical use in providing an infinite source of pseudo-randomness.\footnote{We refer to pseudo-randomness in the sense that an irrational number corresponds one to one to a point of the real line and in this sense is totally deterministic; on the other hand, an irrational number, when proved to be normal, is represented by a sequence of digits with a random behavior.} 

The interest in studying normal numbers lies not only in their randomness but also in the fact that they are extremely difficult to identify and obscure in many other aspects. Despite the appeal of the concept, its trivial interpretation and the proof that almost all real numbers are normal, the proof for given irrational numbers to be normal in some base is still elusive.

The concept of normal number has been proposed in (\cite{nla1}) together with the proof that almost all real numbers are absolutely normal. %Absolutely normality refers to numbers that are shown to be normal in every base b. 
The argument of the proof is not constructive; it is proved that the set of non normal real numbers has Lebesgue measure zero but no proof of normality for some number is given. 
\cite{nla2} provided an alternative proof of Borel's result.
\cite{nla3} proved that the number $C_{10}=0.123456789101112131415161718192021 \dots$, obtained through the concatenation of the positive consecutive integers, is normal in base 10. In general, it is possible to show that the concatenated sequence of positive integers in any base $b \geq 2$ is a b-normal number.
\cite{nla4} proved that the number $0.23571113171923293137 \ldots $
obtained by the concatenation of the prime numbers is normal in base 10. The last two cited results provide constructive approaches directly proving that some given numbers are normal. %Despite the general result of Borel [] ensuring that a random chosen real number is absolutely normal with probability 1 referring to Lebesgue measure, it is difficult to prove the normality of a given number.

More recently, \cite{nla5} proved that if $f$ is such that $f(x)>0$ for $x>0$, then the real number $0.\left\lfloor f(1) \right\rfloor \left\lfloor f(2) \right\rfloor$ $\left\lfloor f(3) \right\rfloor \ldots,$ where $\left\lfloor f(n) \right\rfloor$ is the integer part of $f(n)$ expressed in base $b \geq 10$, is normal in base b. Some extensions of the previous result can be found in \cite{nla6} and \cite{nla7}.

\cite{nla10} suggested to consider series representing numeric constants as good candidates for normality. \cite{nla8} provided a proof for some series to represent normal numbers in base 2; \cite{nla9} used a similar approach to prove an entire uncountable class of numbers to be normal in base 2.

The paper is organized as follows: Section 2 enumerates some general results about normality of binary numbers; Section 3 discusses the normality of $\sqrt{2}$; Section 4 concludes the paper.

\section{General results on normal numbers in base 2} \label{sec1}
The paper discusses about normality of binary numbers; the base is $b = 2$. For sake of simplicity and when there is no possibility of misunderstandings we maintain the decimal notation of the number. For example, we prefer to refer to the binary number $\sqrt{2}$ instead of using the notation $\sqrt{10}$.     

\begin{definition}{[Vector representation of a number]}\label{def1} Given a real number $x$ expressed in binary form and an integer $n>1$, $[x]^{(n)}$ is the column vector of size $n$ containing the first $n$ digits of $x$.
\end{definition}

\begin{remark} \label{rem1} As a consequence of Definition \ref{def1}, the numbers $x_1$ and $x_2 = 10^p x_1$, with $p$ a positive integer, have the same vector representation $[x_1]^{(n)} = [x_2]^{(n)}$. For example: fix $x_1 = 100.0$ and $x_2 = 100.0 \times 10 = 1000$. Then we have $[x_1]^{(2)} = [1 0]' = [x_2]^{(2)}$, $[x_1]^{(3)} = [1 0 0]' = [x_2]^{(3)}$ and, in general, $[x]^{(n)} = [x_1]^{(n)}$.
\end{remark}

Considering that the vector $[x]^{(n)}$ corresponds to an infinite class of numbers identified by the first $n$ digits of their expansion, we propose to choose a specific element of the class in order to identify the entire class.

\begin{definition}{[Integer representative of the class]}\label{rep} Given a real number $x$ expressed in binary form and the integers $n>1$ and $p>1$, the integer number $x^* = \left\langle [x]^{(n)},\textbf{2}^{(n)} \right\rangle$ is chosen as representative of the numbers $2^p x$, where $\textbf{2}^{(n)}$ is the column vector of size $n$ with entries $\textbf{2}^{(n)}_i = 2^{n-i}$, for $i = 1, \ldots, n$.
\end{definition}
The choice of $x^*$ as representative of the numbers with the same vector representation $[x]^{(n)}$ is arbitrary but it will be useful for further calculations.

\begin{remark} \label{rem2} If $x$ is irrational or with an infinite sequence of digits, the vector $[x]^{(n)}$ represents an approximation by truncation of the sequence of digits of $x$; the infinite sequence of digits of $x$ corresponds to $\lim_{n \rightarrow + \infty} [x]^{(n)}$.
\end{remark}

%Moreover, definition \ref{def1} associates to a given irrational number its finite approximation represented by the vector $[x]^{(n)}$; the infinite sequence of digits of $x$ corresponds to $\lim_{n \rightarrow + \infty} [x]^{(n)}$.  

We define $\textbf{1}^{(n)}$ as the unitary column vector of $n$ components and $\alpha$ the angle between the vectors $\textbf{1}^{(n)}$ and $[x]^{(n)}$. Note that, working in base 2, the vector $\textbf{1}^{(n)}$ corresponds to the vector representation of the numbers $(2^n-1)2^p$ with $p$ a positive integer while the integer representative of the class is $2^n-1$.

\begin{theorem}\label{teo1} Given a real number $x$, if the angle between $[x]^{(n)}$ and $\textbf{1}^{(n)}$ is $\alpha = \pm \frac{\pi}{4}$, then 
\[\lim_{n \rightarrow + \infty} \frac{\left\langle [x]^{(n)}, \textbf{1}^{(n)} \right\rangle}{n} = \frac{1}{2}
\]   
(i. e. the numbers $x2^p$ with $p$ a positive integer, represented by the vector $[x]^{(n)}$, are normal in base 2). 
\end{theorem} 
\begin{proof}
By construction the following equality holds:
\[\left\langle [x]^{(n)}, \textbf{1}^{(n)} \right\rangle = \left\langle [x]^{(n)}, [x]^{(n)} \right\rangle
\]
Using the assumption $\alpha = \pm \frac{\pi}{4}$, we obtain
\[\| [x]^{(n)} \| \quad \| \textbf{1}^{(n)} \| \cos \left( \pm \frac{\pi}{4} \right) = \| [x]^{(n)} \| \quad \| [x]^{(n)} \| 
\]
where $\| \cdot \|$ represents the Euclidean norm.
Knowing that
\[\| \textbf{1}^{(n)} \| = \left\langle  \textbf{1}^{(n)}, \textbf{1}^{(n)} \right\rangle^{\frac{1}{2}} = \sqrt{n}  
\]
we obtain
\[ \| [x]^{(n)} \| = \sqrt{n} \frac{\sqrt{2}}{2} = \sqrt{\frac{n}{2}}.
\]
Then
\[\lim_{n \rightarrow + \infty} \frac{\left\langle [x]^{(n)}, \textbf{1}^{(n)} \right\rangle}{n} = \lim_{n \rightarrow + \infty}\frac{\| [x]^{(n)} \| \quad \| \textbf{1}^{(n)} \| \cos \left( \pm \frac{\pi}{4} \right)}{n} = \frac{1}{2}
\]
\end{proof} 
\begin{remark} The result in theorem \ref{teo1} is equivalent to $\lim_{n \rightarrow + \infty} \| [x]^{(n)} \| = +\infty$ and $\lim_{n \rightarrow + \infty} \frac{\| [x]^{(n)} \|}{\sqrt{\frac{n}{2}}} = 1$, i.e. using the binary representation of a number together with Definition \ref{def1}, the square of the Euclidean norm of the vector representation is equivalent to the counting process of the ones in the sequence of the digits. 
\end{remark}

%%QUESTA PARTE VA MESSA IN FONDO ALLA SECTION?***
In general, there is no direct relation between a number $x$ expressed in binary form and the norm of its vector representation. For example, given the vectors $[x_1]^{(3)} = [1 1 0]'$ and $[x_2]^{(3)} = [1 0 1]'$, we have  
\[||[x_1]^{(3)}|| = ||[1 1 0]'|| = \sqrt{2} = ||[1 0 1]'|| = ||[x_2]^{(3)}||.  
\] while $x_1^* = 110 > 101 = x_2^*$.
In order to solve the problem that the order relations between numbers is not guaranteed when working with the norms of their corresponding vector representations, we propose an alternative vector representation.
%\begin{definition}{[Non Standard vector representation of a number]}\label{defns} Given a real number $x$ expressed in binary form and an integer $n>1$, $[x_{ns}]^{(n)}$ is the column vector of size $2^n-1$ such that the $i^{th}$ entry of the vector $[x_{ns}]_i^{(n)} = \textit{1}_{2^{(n-1)}<i \leq \sum_{j=0}^{i-1} 2^{n-j}}  [x]^{(n)}$, where $i = 1, \ldots, n$, $[x]^{(n)}$ is the vector representation from definition \ref{def1} and $\textit{1}_{2^{(n-1)}<i \leq 2^n}$ is the indicator function, i.e. a function equal to $1$ when the argument is in the given interval and $0$ otherwise.
%\end{definition}

\begin{definition}{[Non Standard vector representation of a number]}\label{defns} Given a real number $x$ expressed in binary form and an integer $n>1$, $[x]_{ns}^{(n)}$ is the column vector of size $2^n-1$ obtained concatenating in a unique vector the vectors $[[v]^1, [v]^2, \ldots, [v]^n]$ where $[v]^i = [x]_i^{(n)} \textbf{1}^{\left(2^{n-i}\right)}$ for $i = 1, \ldots, n$ and $[x]_i^{(n)}$ is the $i^{th}$ entry of the vector representation of $x$. %and $\textbf{1}^{2^{n-1}}$ is the unitary vector of size $2^{n-1}$.
\end{definition}
The non standard vector representation permits to directly relate the norm of the vector representation with the corresponding represented number.
%The non standard vector representation creates a situation in which binary numbers and the norm of their vector representations not only maintain the order relation but also a simple proportion. 
\begin{proposition}\label{propnorm} Given a real number $x$ expressed in binary form and an integer $n>1$, $\|[x]_{ns}^{(n)}\| = \sqrt{x^*}$.
\end{proposition}
\begin{proof} The result trivially derives from Definitions \ref{rep} and \ref{defns}.
\end{proof}
Proposition \ref{propnorm} highlights the importance of choosing the number $x^*$ as the representative of the numbers with vector representation $[x]^{(n)}$, see Definition \ref{rep}, in order to link the norm of a vector to the number represented in the vector.
Table \ref{tab1} resumes and exemplify for some numbers their binary representations, the standard and non standard vector representations together with the respective norms. 
The entries of the vector $[x]_{ns}^{(3)}$ in Table \ref{tab1} are represented separated by spaces in order to highlight the concatenated vectors, see Definition \ref{defns}.

%\begin{center} 
\begin{table}[!h] 
\centering
\caption{Example of vector representations}
\label{tab1}
\begin{tabular}{|c|c|c|c|c|}
\hline
$[x]^{(3)}$	& $||[x]^{(3)}||$ &	$[x]_{ns}^{(3)}$ &	$||[x]_{ns}^{(3)}||$ & $x^*$ \\
\hline
$[001]'$	& 1	& $[0000 \quad 00 \quad 1]'$	& 1  & 1 \\
\hline
$[010]'$	& 1	& $[0000 \quad 11 \quad 0]'$	& $\sqrt{2}$ & 2 \\
\hline
$[011]'$ &	$\sqrt{2}$ &	$[0000 \quad 11 \quad 1]'$	& $\sqrt{3}$ & 3 \\
\hline
$[100]'$ &	1	& $[1111 \quad 00 \quad 0]'$	& $\sqrt{4}$ & 4 \\
\hline
$[101]'$ &	$\sqrt{2}$	& $[1111 \quad 00 \quad 1]'$ &	$\sqrt{5}$ & 5 \\
\hline
$[110]'$	& $\sqrt{2}$ &	$[1111 \quad 11 \quad 0]'$	& $\sqrt{6}$ & 6 \\
\hline
$[111]'$	& $\sqrt{3}$	& $[1111 \quad 11 \quad 1]'$	& $\sqrt{7}$ & 7 \\
\hline
\end{tabular}
\end{table} 
%\begin{remark}\label{rem1} Note that, given a positive integer $x$, $||[x]_{ns}|| = \sqrt{x}$. By construction, the number of ones in the non standard vector representation of $x$ is $x$, as shown in Table \ref{tab1}.
%\end{remark}
%\end{center} 

%//sono arrivato fino a qui
\begin{definition}\label{def2} Given a real binary number $x$ and an integer $n>1$, a complement of $x$ is the number $x^c$ such that $[x]^{(n)} + [x^c]^{(n)} = [2^n-1]^{(n)}$.%% is defined as the complement of $x$\footnote{We recall that the number $2^n-1$ is the sequence of ones of length $n$ in binary form and $[2^n-1] = \textbf{1}^{(n)}$}. 
\end{definition}
\begin{remark} Note that the complement of a given number is not unique due to the multiple correspondence between binary numbers and their vector representation.\end{remark}
The following result links numbers represented in binary form and the norm of their vector representation.
\begin{proposition}\label{norme} Given a natural number $n$ and a real binary number $x$, then $\sqrt{\|[x]^{(n)}\|^2+\|[x^c]^{(n)}\|^2} = \| \textbf{1}^{(n)} \| = \sqrt{n}$.  
\end{proposition}
\begin{proof} Considering Definition \ref{def2}, the vectors representing $x$ and $x^c$,respectively $[x]^{(n)}$ and
$[x^c]^{(n)}$ are orthogonal by construction:
\[\left\langle [x]^{(n)}, [x^c]^{(n)} \right\rangle = 0.\]
The result of the proposition follows directly from this property.
\end{proof}
\begin{remark} The result in Proposition \ref{norme} holds substituting the vector representation with the non standard vector representation. Given a natural number $n$ and a real binary number $x$, then $\sqrt{\|[x]_{ns}^{(n)}\|^2+\|[x^c]_{ns}^{(n)}\|^2} = \| \textbf{1}_{ns}^{(n)} \| = \sqrt{2^{n}-1}$, where $\textbf{1}_{ns}^{(n)}$ is the non standard vector representation of the binary number $2^{n}-1$, i.e. a vector with $2^{n}-1$ unitary entries. 
\end{remark}

The previous results permit to state a general condition for normality of numbers in base 2.
\begin{proposition}\label{proppp} The irrational binary number $x$ is normal in base 2 iif
\[\lim_{n \rightarrow + \infty}\| [x]^{(n)} \| = \lim_{n \rightarrow +\infty} \| [x^c]^{(n)} \|
\]
\end{proposition}
\begin{proof} Considering that $[x]^{(n)}$ and $[x^c]^{(n)}$ are orthogonal by construction, the ones of $[x]^{(n)}$ correspond to the zeros of $[x^c]^{(n)}$ and viceversa. If $[x]^{(n)}$ and $[x^c]^{(n)}$ have definitively the same norm, for $n \rightarrow +\infty$ the number of ones in $[x]^{(n)}$ is equivalent to the numbers of ones in $[x^c]^{(n)}$, i.e. the number of zeros in $[x]^{(n)}$; this permits conclude that $[x]^{(n)}$ is normal in base 2. The same argument holds for the normality of $[x^c]^{(n)}$. 

If $x$ is normal in base 2 the number of ones in the sequence of its digits is definitively equal to the numbers of zeros. As a consequence, for $n \rightarrow +\infty$ the number of ones in the vector $[x]^{(n)}$ is equal to the number of zeros and therefore to the number of ones of $[x^c]^{(n)}$. For $n \rightarrow +\infty$ the two vectors $[x]^{(n)}$ and $[x^c]^{(n)}$ have the same norm.  
\end{proof}

\begin{corollary} The irrational binary number $x^c$ is normal in base 2 iif $x$ is normal in base 2.
\end{corollary}

The result in Proposition \ref{proppp} still holds substituting the vector representation with the non standard vector representation. The two situations differ for the interpretation: using the vector representation of a number, the condition expressed in Proposition \ref{proppp} permits to prove the normality of the number itself; on the other hand, when using the non standard vector representation, Proposition \ref{proppp} provides a condition of normality holding only for the binary sequence of the vector that corresponds to the binary expansion of some unknown number. In other words, if we conclude that $\lim_{n \rightarrow + \infty}\| [x]_{ns}^{(n)} \| = \lim_{n \rightarrow +\infty} \| [x^c]_{ns}^{(n)} \|$ we prove that the binary sequences in the vectors are normal without knowing which numbers correspond to the sequences. 

In practice, if the condition of proposition \ref{proppp} is satisfied for a given vector, when can refer to its normality intending that all the numbers corresponding to that vector are normal in base 2. 
The following result links the normality of $[x]_{ns}^{(n)}$ to the normality of $[x]^{(n)}$.

\begin{theorem}\label{teone} Given a real binary number $x$, if $[x]_{ns}{(n)}$ is normal then $x$ is normal.
\end{theorem} 
\begin{proof} If the vector $[x]_{ns}^{(n)}$ is normal, the following holds:
\[\lim_{n \rightarrow +\infty} \frac{\| [x]_{ns}^{(n)} \|}{\sqrt{\frac{2^n}{2}}} = \lim_{n \rightarrow +\infty} \frac{\| [x^c]_{ns}^{(n)} \|}{\sqrt{\frac{2^n}{2}}} = 1.
\]
For $n \rightarrow +\infty$ the number of ones tend to equal the number of zeros; we apply a permutation $p$ changing the order of the entries of $[x]_{ns}^{(n)}$ and call $p\left([x]_{ns}^{(n)}\right)$ the reordered vector. Since no permutation is able to change the norm of a vector, we have that $p\left([x]_{ns}^{(n)}\right)$ is normal for all permutations $p$. We choose the permutation $\bar{p}$ such that the vectors $[v]^1, \ldots, [v]^n$ concatenated in $[x]_{ns}^{(n)}$ have the same proportions of zeros and ones. (Note that the permutation $\bar{p}$ is not unique). Consider the ordered sequences obtained drawing the $i^{th}$ element of the sequence from the vector $\bar{p}([v]^i)$; the probability of having zero/one in position $i$ in the ordered sequence is equal to $\frac{1}{2}$. As a consequence, every ordered sequence is normal in base 2. By construction, the sequence represented by the elements of vector $[x]^{(n)}$ for $n \rightarrow +\infty$ is one of the possible ordered sequences.  

%At this point, recalling the relation between $[x]_{ns}^{(n)}$ and $[x]^{(n)}$, see Definition \ref{defns}, each entry $i$ of the vector $[x]$ can be interpreted as drawn from the vector $\bar{p}([v]^i)$, where $\bar{p}([v]^i)$ is the vector $[v]^i$ reordered following the permutation $\bar{p}$. This concludes the proof because we have shown that the sequence of digits of $[x]$ is equivalent to the sequence of heads and tails when flipping a coin.  
\end{proof}

\section{Normality of $\sqrt{2}$ in base 2}

Set $x = \sqrt{2}$ as the binary expansion of $\sqrt{2}$ and $[x]^{(n)} = [\sqrt{2}]^{(n)}$ its approximation by truncation at the first $n$ digits in the vector representation. Following Definition \ref{def1} and Remark \ref{rem1}, the vector representation $[\sqrt{2}]^{(n)}$ corresponds to the vector representation of the number $2^{n-1} \sqrt{2}$ truncated after the first $n$ digits. Note that, given $n$, $[2^{n-1}\sqrt{2}]^{(n)}$ is the vector representation of an integer number sharing with the $\sqrt{2}$ the sequence of the digits till the truncation.

%In the previous section we showed through a trivial example that given two numbers $x_1$ and $x_2$ such that $x_1 > x_2$ we can not conclude in general that $\|[x_1]\| > \|[x_2]\|$. %In words, the order relation between numbers does not guarantee the same order relation on the norm of the corresponding vector representations. 

%The idea is to start from a number $x$ expressed in base 2, then pass to its $[x]^{(n)}$ and $[x]_{ns}^{(n)}$ vector representations. %The first representation $[x]^{(n)}$ has a direct relation to the normality of $x$ in base 2 through its norm as showed in section \ref{sec1}. The second representation $[x]_{ns}^{(n)}$ permits to preserve the order relation passing from the numbers to the norm of their vector representations. In this way, we can conclude something about the norm of $[x]_{ns}^{(n)}$ and considering the relation with $[x]^{(n)}$ we then conclude something about the normality of $x$.

%We consider the binary number $2^{n-1} \sqrt{2}$. This number shares with $\sqrt{2}$ the sequence of digits and, following definition \ref{def1}, the vector representation. 
We define the vector representation of $2^{n-1} \sqrt{2}$ as $[2^{n-1} \sqrt{2}]^{(n)}$ and the non standard vector representation as $[2^{n-1} \sqrt{2}]_{ns}^{(n)}$.

Considering the result in Proposition \ref{propnorm}, we can calculate $\|[2^{n-1} \sqrt{2}]_{ns}^{(n)}\|$ following the following proportion: 
\[
2^n-1 : 2^{n-1}\sqrt{2} = \|[2^n-1]_{ns}^{(n)}\| : \|[2^{n-1} \sqrt{2}]_{ns}^{(n)}\|
\]
where, by construction, $\|[2^n-1]_{ns}^{(n)}\| = \sqrt{2^n-1}$ and, for large values of $n$, $\|[2^n-1]_{ns}^{(n)}\| \approx \sqrt{2^n}$.
Solving the proportion and calculating the limit for $n \rightarrow +\infty$ in order to consider the infinite sequence of the digits in the binary expansion we obtain:
\begin{equation}
\lim_{n \rightarrow +\infty} ||[2^{n-1} \sqrt{2}]_{ns}^{(n)}|| = +\infty \label{ris0}
\end{equation}
and
\begin{equation}
\lim_{n \rightarrow +\infty} \frac{||[2^{n-1} \sqrt{2}]_{ns}^{(n)}||}{\sqrt{\frac{2^n}{2}}} = \lim_{n \rightarrow +\infty} \frac{\frac{2^{n-1}\sqrt{2} \quad \sqrt{2^n-1}}{2^n-1}}{\sqrt{\frac{2^n}{2}}} = 1  \label{ris2}
\end{equation}
If we consider the results (\ref{ris0}) and (\ref{ris2}) together with Proposition \ref{norme} we obtain that 
\[
\lim_{n \rightarrow +\infty} ||[(2^{n-1} \sqrt{2})^c]_{ns}^{(n)}|| = \sqrt{\frac{2^n}{2}} = \lim_{n \rightarrow +\infty} ||[2^{n-1} \sqrt{2}]_{ns}^{(n)}||
\]
ensuring that the norm of the non standard vector representations of $[2^{n-1} \sqrt{2}]_{ns}^{(n)}$ and its complement have the same norm for $n \rightarrow +\infty$. For Proposition \ref{proppp} the two numbers represented by the vectors $[(2^{n-1} \sqrt{2})^c]_{ns}^{(n)}$ and $[2^{n-1} \sqrt{2}]_{ns}^{(n)}$ are normal in base $b = 2$.
Considering the results in Theorem \ref{teone} and Remark \ref{rem1} the number $x = \sqrt{2}$ is normal in base 2.
%We note that the numbers proved to be normal in base $b = 2$ have a strong relation with the number $\sqrt{2}$, i.e. the sequence of digits in the vectors $[2^{n-1} \sqrt{2}^c]_{ns}^{(n)}$ and $[2^{n-1} \sqrt{2}]_{ns}^{(n)}$... We can define a new irrational number $\sqrt{u}$ such that:
%\[ [\sqrt{u}]^{(n)} = [2^{n-1} \sqrt{2}]_{ns}^{(n)} \quad \forall n.
%\]
%The number $\sqrt{u}$ is proved to be normal in base $b = 2$.

In order to explain the argumentation we technically show the idea through a toy example. 

\begin{example}\label{ex1}
We set 
\[x = \sqrt{2} = 1.01101010000010011110 \ldots
\]    
and $n = 4$ for simplicity.

Following the Definitions \ref{def1} and \ref{defns}, we have
\[[x]^{(4)} = [\sqrt{2}]^{(4)} = [1011] = [2^3 \sqrt{2}]^{(4)} \quad \quad [x]_{ns}^{(4)} = [11111111 \quad 0000 \quad 11 \quad 1]%\footnote{The spaces in the vector representation simplify the understanding of the passage from the binary expansion to the relative non standard vector representation.}
\]
Considering the proportion between the norm of non standard vector representation and the number presented in Table \ref{tab1} and in Remark \ref{rem1}, we have that $||[x]_{ns}^{(4)}|| = \sqrt{11}$.
Note that in this example 
\[[x^c]_{ns}^{(4)} = [00000000 \quad 1111 \quad 00 \quad 0] \quad and \quad ||[x^c]_{ns}^{(4)}|| = \sqrt{4}
\] 
while
\[[2^n-1]_{ns}^{(4)} = [11111111 \quad 1111 \quad 11 \quad 1] \quad and \quad ||[2^n-1]_{ns}^{(4)}|| = \sqrt{15} 
\]
%\end{example}

Passing to the limit for $n \rightarrow +\infty$, first we prove the normality in base 2 of the number $[x]_{ns}^{(n)}$, then we choose a permutation $\bar{p}$ as described in Theorem \ref{teone} obtaining what follows:
\[\lim_{n \rightarrow +\infty} \bar{p}([x]_{ns}^{(n)}) = [111 \ldots 111 000 \ldots 000 \quad 00 \ldots 00 11 \ldots 11 \quad \ldots \quad 1 \ldots 1 0 \ldots 0 \quad \ldots]
\]
The notation highlights that each vector $\bar{p}([v]^i)$ is rebalanced in terms of zeros and ones (representing a fair coin) but with a different number of elements.  

%we investigate the normality of $\sqrt{2}$ through its relations to $\sqrt{u}$. The non standard vector representation of $\sqrt{2}$ has proven to be normal but the vector with the property of normality corresponds to an unknown number $\sqrt{u}$. In order to conclude something on the normality of $\sqrt{2}$ we introduce a third representation of the number named triangular representation.

\end{example}

\section{Conclusion} In this paper we provide a technique to investigate normality of irrational numbers in base 2. An infinite class of number is proved to be normal. The result is interesting considering the difficulty to build normal numbers and/or prove their normality. We also show that $\sqrt{2}$ is normal in base 2. Despite the approach described in Section 2 is general and suitable in power for testing normality of a generic irrational numbers, it is not clear if it can be useful to investigate the normality of well known irrational numbers conjectured to be normal as, for example, $\pi$ or $e$. The difficulty to prove the normality of a given irrational binary number relies on the randomness of its sequence of digits; the result proved in the paper uses the fact that normal numbers, i.e. numbers showing a random behavior in the digital expansion, correspond one to one in a deterministic way to specific points on the real line. %(The behavior is random but the position on the real line is deterministic permitting to conclude something on the sequence of digits of the number)


\newpage

\begin{thebibliography}{99}

%\bibitem[\protect\citeauthoryear{Allaire and Kaber}{2008}]{nla1}Allaire, G. and Kaber, S.\ M. (2008). Numerical linear algebra (Vol. 55). New York: Springer.

\bibitem[\protect\citeauthoryear{Bailey $\&$ Crandall}{2001}]{nla9}Bailey, D.H., Crandall, R.E. (2001) On the random character of fundamental constant expansions. Experimental Mathematics. 10(2), 175-190 

\bibitem[\protect\citeauthoryear{Borel}{1909}]{nla1}Borel, E. (1909) Les probabilites d\'enombrables et leurs applications arithmetiques. Rendiconti del Circolo Matematico di Palermo 27, 247-271

\bibitem[\protect\citeauthoryear{Champernowne}{1933}]{nla3}Champernowne, D.G. (1933) The construction of decimals normal in the scale of ten. Journal of London Mathematical Society 8, 254-260 

\bibitem[\protect\citeauthoryear{Copeland $\&$ Erdos}{1946}]{nla4}Copeland, A.H., Erdos, P. (1946) Note on normal numbers. Bulletin of the American Mathematical Society 52, 857-860 

\bibitem[\protect\citeauthoryear{Madritsch et al.}{208}]{nla7}Madritsch, M.G., Thuswaldner, J.M., Tichy, R.F. (2008) Normality of numbers generated by the values of entire functions. Journal of Number Theory 128, 1127-1145 

\bibitem[\protect\citeauthoryear{Nakai $\&$ Shiokawa}{1992}]{nla5}Nakai, Y., Shiokawa, I. (1992) Discrepancy estimates for a class of normal numbers. Acta Arithmetica. 62(3), 271-284 

\bibitem[\protect\citeauthoryear{Nakai $\&$ Shiokawa}{1997}]{nla6}Nakai, Y., Shiokawa, I. (1997) Normality of numbers generated by the values of polynomials at primes. Acta Arithmetica. 81(4), 345-356 

\bibitem[\protect\citeauthoryear{Sierpinski}{1917}]{nla2}Sierpinski, W. (1907) Demonstration elementaire du theoreme de M. Borel sur les nombres absolument normaux et determination effective d'un tel nombre. Bulletin de la Societe Mathematique de France, 45, 127-132 

\bibitem[\protect\citeauthoryear{Stoneham}{1971}]{nla10} Stoneham, R. (1971)] A General Arithmetic Construction of Transcendental Non-Liouville Normal Numbers from Rational Fractions. Acta Arithmetica. 16, 239-253

\bibitem[\protect\citeauthoryear{Stoneham}{1973}]{nla8}Stoneham, R. (1973) On absolute (j,$\epsilon$)-normality in the rational fractions with applications to normal numbers. Acta Arithmetica. 22, 277-286 



\end{thebibliography}
\end{document}